\documentclass[12pt]{article}

 	\usepackage{amsfonts,amssymb,amsthm,amsmath,cite,bbm}
	\newcommand{\R}{\mathbb R} 
	\newcommand{\one}{\operatorname{\mathbbm{1}}}
	
	\newcommand{\om}{\operatorname{\rm Z}} 
	\newcommand{\omd}{\operatorname{{\rm Z}}^*\!}
	\newcommand{\rbd}{\partial^*} 
	\newcommand{\diam}{\operatorname{diam}}
	\newcommand{\un}{\nu}
		
	\newcommand{\sn}{{S^{n-1}}}
	\newcommand{\affn}{\operatorname{Aff}(n,1)}
	\newcommand{\rn}{{\R^{n}}}

	\newcommand{\card}{H^0}

	\newcommand{\BVZK}[1]{\|#1\|_{BV,\, \om K}}

	\providecommand{\norm}[1]{\lVert#1\rVert}
	\providecommand{\bnorm}[1]{\Big\lVert#1\Big\rVert}
	
	\providecommand{\abs}[1]{\lvert#1\rvert}
	\providecommand{\babs}[1]{\Big\lvert#1\Big\rvert}
	\providecommand{\Vol}[1]{\lvert#1\rvert}

	\newcommand{\eqnref}[1]{(\ref{#1})}

	\newtheorem{theorem}{Theorem}
	\newtheorem{lemma}[theorem]{Lemma}
	\newtheorem{corollary}[theorem]{Corollary}

\title{Anisotropic Fractional Perimeters}
\author{Monika Ludwig\footnote{The work of the author was supported, in part, by Austrian Science Fund (FWF) Project P23639-N18.}}

\date{}

\begin{document}

\maketitle
\begin{abstract}
The anisotropic fractional $s$-perimeter with respect to a convex body $K$ in $\rn$ is shown to converge as $s\to 1^-$ to the anisotropic perimeter with respect to the moment body of $K$. 
For anisotropic fractional {$s$-seminorms} on $BV(\rn)$, the corresponding result is established (generalizing  results of  Bourgain, Brezis \& Mironescu and D\'avila).  Minimizers of the anisotropic fractional {iso\-perimetric} inequality with respect to $K$ are shown to converge to the moment body of $K$ as $s\to 1^-$.  Anisotropic fractional Sobolev inequalities are established.
\end{abstract}

\bigskip

For a Borel set $E\subset \rn$ and $0<s<1$, the fractional $s$-perimeter of $E$ is given by
$$P_s(E) = \int_E\int_{E^c} \frac1 {\abs{x-y}^{n+s}} \,dx\,dy,$$
where $E^c$ denotes the complement of $E$ in $\rn$ and $\abs{\,\cdot\,}$ the Euclidean norm on $\rn$.
Fractional peri\-meters are closely related to fractional Sobolev seminorms (see Sections \ref{norms} and \ref{sobineq}). 
On Borel sets in $\rn$, the functional $P_s$ is an $(n -s)$-dimensional perimeter, as $P_s(\lambda E) = \lambda^{n -s} P_s(E)$ for $\lambda > 0$. It is non-local in the sense that it is not determined by the behavior of $E$ in a neighborhood of $\partial E$. 
Fractional $s$-perimeters have attracted increased attention in recent years (see \cite{AmbrosioDePhilippisMartinazzi, BourgainBrezisMironescu02, Visintin_siam, FuscoMillotMorini,FrankSeiringer,  CaffarelliRoquejoffreSavin, CaffarelliValdinoci, DipierroFigalliPalatucciValdinoci} and the references therein). 

The limiting behavior of fractional $s$-perimeters as $s\to 1^-$ and  as $s\to 0^+$ turns out to be very interesting. 
A result by D\'avila \cite{Davila}, which extends results by Bourgain, Brezis \& Mironescu \cite{BourgainBrezisMironescu},  shows that for a bounded Borel set $E\subset\rn$ of finite perimeter,
\begin{equation}\label{davila}
\lim_{s\to1^-} (1-s)  P_s(E) = \alpha_{n} P(E),
\end{equation}
where $P(E)$ is the perimeter of $E$ and $\alpha_n$ is a constant depending on $n$. The peri\-meter $P(E)$ coincides with the $(n - 1)$-dimensional Hausdorff measure of $\partial E$ when $E$ has  smooth boundary. If  $E$ is a Borel set of
finite Lebesgue measure, then $E$ is of finite perimeter if its characteristic function $\one_E$ is in $BV(\rn)$ and 
 then $P(E)$ is the total variation of the weak gradient of $\one_E$. We refer to \cite{AmbrosioFuscoPallara, Maggi} for the basic
properties of sets of finite perimeter and note that
\begin{equation}\label{peri}
P(E) =\int_{\rbd E} \abs{\nu_E(x)} \,dH^{n-1}(x),
\end{equation}
where $H^{n-1}$ denotes $(n-1)$-dimensional Hausdorff measure, $\rbd E$  the reduced boundary of $E$ and 
$\un_E(x)$ the measure theoretic outer unit normal  of $E$ at $x\in\rbd E$. If $E$ has  smooth boundary, then $\rbd E$ is just the topological boundary, $\partial E$, of $E$ and $\un_E(x)$  is the usual outer unit normal vector of $E$ at $x\in\partial E$.

The limiting behavior for $s\to 0^+$ of fractional Sobolev $s$-seminorms was determined by Maz$'$ya \& Shaposhnikova \cite{MazyaShaposhnikova}. Their result implies that 
\begin{equation}\label{mazyashaposhnikova}
\lim_{s\to 0+}s\,P_s(E)= n\, \Vol{B} \,\Vol{E},
\end{equation}
for every bounded Borel set $E\subset \rn$ of finite fractional $s'$-perimeter for all $s'\in(0,1)$. Here $B$ is the Euclidean unit ball and $\Vol{\,\cdot\,}$ is the $n$-dimensional Lebesgue measure.  See Dipierro, Figalli, Palatucci \& Valdinoci \cite{DipierroFigalliPalatucciValdinoci} for a detailed study of the limiting behavior in this case.

Anisotropic perimeter is  a natural generalization of the Euclidean notion of perimeter obtained by replacing the Euclidean norm $\abs{\,\cdot\,}$ in \eqnref{peri} by an arbitrary norm $\norm{\,\cdot\,}_L$ with unit ball $L$. We say that a set  $K\subset\rn$ is a convex body if it is  compact  and convex and has non-empty interior.  For $K\subset \rn$ an origin-symmetric convex body, the anisotropic perimeter of a Borel set $E\subset \rn$ with respect to $K$ is 
$$P(E, K) = \int_{\rbd E} \norm{\nu_E(x)}_{K^*} \,dx,$$
where $K^*=\{v\in\rn:v\cdot x\le1 \hbox{ for all } x\in K\}$ is the polar body of $K$. 
If $E$ is a convex body, then $P(E,K)$ is equal (up to a factor $n$) to the classical first mixed volume of $E$ and $K$ (cf.~\cite{Gruber, Schneider:CB}). Anisotropic perimeters are important as a model for surface tension in the study of equilibrium configurations of crystals  and constitute the basic model
for surface energies in phase transitions (cf.  \cite{FigalliMaggiPratelli10} and the references therein).
Anisotropic perimeters correspond to anisotropic  Sobolev seminorms which have been studied in numerous papers (cf.\ \cite{AlvinoFeroneTrombettiLions,Cordero:Nazaret:Villani, Gromov_Sobolev, FigalliMaggiPratelli} and the references therein).

For a Borel set $E\subset \rn$, an origin-symmetric convex body $K\subset\rn$  and $0<s<1$, the anisotropic fractional $s$-perimeter of $E$ with respect to $K$ is given by
$$P_s(E, K) = \int_E\int_{E^c} \frac1 {\norm{x-y}_K^{n+s}}\, dx\,dy,$$
where $\norm{\,\cdot\,}_K$ denotes the norm with unit ball $K$.

A natural question is to study the limiting behavior of anisotropic $s$-perimeters as $s\to 1^-$ and $s\to0^+$.  While one might suspect  that the limit as $s\to 1^-$ of anisotropic $s$-perimeters with respect to the origin-symmetric convex body $K$ is the anisotropic perimeter with respect to the same convex body, this turns out not to be true in general.
In Section~\ref{lb}, we show that for $E\subset\rn$ a bounded Borel set of finite perimeter,
$$\lim_{s\to 1^-} (1-s) P_s(E,K) = P(E, \om K).$$
Here the convex body $\om K$ is the moment body of $K$, that is, the convex body such that for $v\in\rn$,
\begin{equation}\label{moment}
\norm{v}_{\omd K} =  \frac{n+1}2 \,\int_K \abs{v\cdot x}\,dx,
\end{equation}
where $\omd K$ is the polar body of  $\om K$. For the Euclidean unit ball $B$, the convex body $\om B$ is just a multiple of  $B$. Hence we  recover \eqnref{davila}  including the value of the constant $\alpha_{n}$.

The moment body is closely related to the classical centroid body of $K$, which is defined as 
$$\frac2{(n+1) \Vol{K}}\, \om K.$$
If we intersect the origin-symmetric convex body $K$ by all halfspaces ortho\-gonal to $u\in\sn$, then the centroids of these intersections  trace out the boundary of twice the centroid body of $K$, which explains the name centroid body. The name moment body comes from the fact that the moment vectors of these halfspaces trace out the boundary (of a constant multiple) of $\om K$. Centroid bodies play an important role in the geometry of convex bodies (cf.\ \cite{Gardner, Lutwak90}) and moment bodies in the theory of valuations on convex bodies (see \cite{Ludwig:Minkowski, Ludwig:convex, Haberl_sln}). In recent years, centroid bodies have found powerful extensions within the $L^p$~Brunn Minkowski
 theory \cite{LZ1997, LYZ2000, LYZ2000b, LYZ2002}, the asymmetric $L^p$~Brunn Minkowski theory  \cite{Haberl:Schuster1, Ludwig:Minkowski} and the Orlicz Brunn Minkowski theory \cite{LYZ2010a}.

In Section \ref{lb}, we  show that for $E\subset\rn$ a bounded Borel set of finite perimeter,
$$\lim_{s\to 0^+} s\,P_s(E, K) = n\, \Vol{K}\, \Vol{E}.$$
The special case when $K$ is the Euclidean unit ball follows from the result \eqnref{mazyashaposhnikova} by Maz$'$ya \& Shaposhnikova. The limiting results for $s\to1^-$ and $s\to 0^+$  for the anisotropic $s$-perimeters are both obtained by using the Blaschke-Petkantschin Formula from integral geometry and results on fractional perimeters for subsets of the real line.

One of the most important results for Euclidean $s$-perimeters is the Euclidean fractional isoperimetric inequality.
For a bounded Borel set $E\subset \rn$,
\begin{equation}\label{fraciso}
P_s(E) \ge \gamma_{n,s}\, \Vol{E}^{\frac{n-s}n}, 
\end{equation}
where $\Vol{E}$ is the $n$-dimensional Lebesgue measure of $E$ and  $\gamma_{n,s}>0$ is a constant depending on $n$ and $s$.  Using a symmetrization result by Almgren \& Lieb \cite{AlmgrenLieb}, Frank \& Seiringer~\cite{FrankSeiringer}  proved that there is equality in \eqnref{fraciso} precisely for balls (up to sets of measure zero). A stability version was recently established by Fusco, Millot \& Morini \cite{FuscoMillotMorini}.  

While it is not difficult to see that for a given origin-symmetric convex body $K$,
there is an optimal constant $\gamma_{s}(K)>0$ such that
\begin{equation}\label{afraciso}
P_s(E, K) \ge \gamma_{s}(K)\, \Vol{E}^{\frac{n-s}n}
\end{equation}
for every bounded Borel set $E\subset \rn$, it turns out that the determination of the minimizers is more challenging and remains  open. Inequality \eqnref{afraciso}  is the anisotropic fractional isoperimetric  inequality. 
In Section~\ref{isoperi}, we show that if minimizers of \eqnref{afraciso} converge to a bounded Borel set $E_1$ as $s\to 1^-$, then $E_1$ is (up to a constant factor) the moment body of $K$.

In the last  sections, we establish  analogues of the results on anisotropic fractional perimeters in the setting of fractional Sobolev spaces. We prove results on the limiting behavior of anisotropic fractional Sobolev seminorms on $BV(\rn)$ and  anisotropic fractional Sobolev inequalities with the sharp constants from \eqnref{afraciso}.

\section{Preliminaries}\label{tools}

We state the Blaschke-Petkantschin Formula (cf.~\cite[Theorem~7.2.7]{SchneiderWeil}) in the case in which it will be used. Let $H^k$ denote the $k$-dimensional Hausdorff measure on $\rn$ and $\affn$ the affine Grassmannian of lines in $\rn$. 
 If $g:\rn\times\rn\to [0,\infty)$ is measurable,
then
\begin{equation}\label{BP}
\begin{array}{l}
\displaystyle\int_\rn\int_\rn g(x,y)\,dH^n(x)\,dH^n(y) \\[12pt]
= 
\displaystyle\int_{\affn} \int_L \int_L g(x,y)\,\abs{x-y}^{n-1}\,dH^1(x)\,dH^1(y)\,dL,
\end{array}
\end{equation}
where $dL$ denotes integration with respect to a suitably normalized, rigid motion invariant Haar measure on $\affn$. This measure can be described in the following way. Any line $L\in\affn$ is parameterized using one of its direction normal vectors $u=u(L)\in\sn$ and its base point $x\in u^\bot$, where $u^\bot$ is the hyperplane orthogonal to $u$, as $L=\{x+\lambda \, u(L): \lambda\in\R\}$. For $h:\affn\to[0,\infty)$ measurable,
\begin{equation}\label{ag}
\int_{\affn} h(L)\,dL= \frac 12 \int_\sn \int_{u^\bot} h(x+L_u)\,dH^{n-1}(x)\,dH^{n-1}(u),
\end{equation}
where $L_u=\{\lambda u:\lambda\in\R\}$.

If $E\subset \rn$ has finite perimeter, then De Giorgi's Structure Theorem (cf.\ \cite[Theorem~15.9]{Maggi})  implies that the reduced boundary, $\rbd E$,  of $E$  is $H^{n-1}$-rectifiable. Hence Theorem~1 of Wieacker \cite{Wieacker84} gives the following:
If $E\subset \rn$ has finite perimeter, then
\begin{equation}\label{proj}
\int_{\rbd E} \abs{u\cdot \un_E(x)}\, dH^{n-1}(x) = \int_{E|u^\bot} \card( \rbd E \cap (y+ L_u))\, dH^{n-1}(y)
\end{equation}
for all $u\in\sn$, where $L_u$ is the line with direction vector $u$. Wieacker \cite{Wieacker84} used the right-hand side of \eqnref{proj} to define the support function of the projection body of $\rbd E$. We remark that Tuo Wang \cite{Tuo_Wang} has defined the projection body of the set $E$ using the left hand side of \eqnref{proj} and obtained the Petty projection inequality for sets of finite perimeter (generalizing a result of Gaoyong Zhang \cite{Zhang99}).

Let $\Omega\subset \rn$ be an open bounded set and let $0<s<1$. We set
\begin{equation}\label{fracisog}
\gamma_{s}(K) =\inf\{  P_s(E, K) \,\Vol{E}^{-\frac{n-s}n}: E\subset \Omega, \Vol{E}>0\}.
\end{equation}
Let $c_1, c_2>0$ be chosen such that $c_1 \le \norm{u}_K \le c_2$  for all $u\in \sn$. Then 
$$c_2^{-(n+s)} P_s(E) \le P_s(E, K) \le c_1^{-(n+s)}P_s(E).$$
Since the optimal constant in the Euclidean fractional $s$-isoperimetric inequality \eqnref{fraciso} is positive, we see that $0<\gamma_s(K)<\infty$ for all origin-symmetric convex bodies $K$. 
Let $E_i\subset \Omega$ be Borel sets such that 
$$\gamma_{s}(K) =\lim_{i\to\infty} P_s(E_i, K) \,\Vol{E_i}^{-\frac{n-s}n}.$$
It follows from 
\cite[(4)]{AmbrosioDePhilippisMartinazzi} (and hence from the Fr\'echet-Kolmogorov compactness criterion) that 
the sequence $E_i$ is pre-compact. In particular, the infimum in \eqnref{fracisog}
is attained. Note that the homogeneity of $P_s(\cdot,K)$ implies that $\gamma_s(K)$ does not depend on the choice of the open bounded set $\Omega$.

Let $E\subset \rn$ be a Lebesgue measurable  set  with $\Vol{E}<\infty$ and $K\subset\rn$ a convex body. The anisotropic isoperimetric  inequality, also called  generalized Minkowski inequality or Wulff inequality, states that
\begin{equation}\label{minkowski}
P(E,K) \ge n \,\Vol{K}^{\frac1n}\,\Vol{E}^{\frac{n-1}n}
\end{equation}
with equality if and only if $E$ is homothetic to $K$ (up to a set of measure zero). If $E$ is a convex body, \eqnref{minkowski} is the classical Minkowski inequality (cf.~\cite{Gruber, Schneider:CB}). For general $E$, inequality \eqnref{minkowski} including the equality case is due to Taylor \cite{Taylor1978}. A quantitative version was recently established by Figalli, Maggi \& Pratelli \cite{FigalliMaggiPratelli10}.

\section{Fractional Perimeters on the Real Line}\label{dimone}

In next lemma, the one-dimensional case of \eqnref{davila} is proved together with estimates, which are used in the proof of Theorem \ref{perimeter}. For a set $A\subset \rn$, let $\diam(A) =\sup\{\abs{x-y}: x\in A, y\in A\}$ denote the (Euclidean) diameter of $A$.

\begin{lemma}\label{one} If $A\subset \R$ is a bounded Borel set of finite perimeter, then
\begin{equation}\label{conv}
\lim_{s\to 1^-}(1-s) P_s(A)= \card(\rbd A)
\end{equation}
and
\begin{equation}\label{domconv}
(1-s) P_s(A) \le 8\,\card(\rbd A) \max\{1,\diam(A)\}
\end{equation}
for $\,1/2\le s<1$.
\end{lemma}

\begin{proof} Since $A$ has finite perimeter, it is up to a set of measure zero the disjoint union of finitely many intervals lying at mutually positive distance (cf.~\cite[Proposition 12.13]{Maggi}). Also note that $\rbd A$  is not changed by changing $A$ on a set of measure zero (cf.~\cite[Remark 15.2]{Maggi}). 
Hence, w.l.o.g., we write $A= \bigcup_{i=1}^m I_i$, where $I_i=(a_i,b_i)$. Set $J_j= (b_j, a_{j+1})$ for $j=1,\dots, m-1$ and $J_0=(-\infty, a_1)$ and $J_m=(b_m,\infty)$. Hence
\begin{equation}\label{sum}
P_s(A)= \sum_{j=0}^m\, \int_{J_j} \int_A  \frac{1}{\abs{x-y}^{s+1}}  \, dx\,dy .
\end{equation}
A simple calculation shows that
\begin{equation*}
\lim_{s\to1^-} (1-s) \int_{J_j} \int_A  \frac{1}{\abs{x-y}^{s+1}}  \, dx\,dy =1
\end{equation*}
for $j=0$ and $j=m$ and
\begin{equation*}
\lim_{s\to1^-} (1-s) \int_{J_j} \int_A  \frac{1}{\abs{x-y}^{s+1}}  \, dx\,dy =2
\end{equation*}
for $j=1, \dots, m-1$.
Since $\card(\rbd A)= 2\,m$, we obtain \eqnref{conv} from \eqnref{sum}.

We have
\begin{eqnarray*}
 \int_{J_0} \int_A  \frac{1}{ \abs{x-y}^{s+1}}  \, dx\,dy &\le& 
\int_{-\infty}^{a_1} \int_{a_1}^{b_m} \frac{1}{(x-y)^{s+1}}  \, dx\,dy \\
&\le& \frac2{1-s} \,\max\{1,\diam(A)\}
\end{eqnarray*}
and similarly
$$\int_{J_m} \int_A  \frac{1}{\abs{x-y}^{s+1}}  \, dx\,dy \le\frac2{1-s} \,\max\{1,\diam(A)\}.$$
For $j=1, \dots, m-1$, 
we have
\begin{eqnarray*}
 \int_{J_j} \int_A  \frac{1}{\abs{x-y}^{s+1}}  \, dx\,dy &\le& 
  \int_{a_1}^{b_j} \int_{b_j}^{a_{j+1}}  \frac{1}{(y-x)^{s+1}}  \, dy\,dx \\
&&+ \int_{a_{j+1}}^{b_m} \int_{b_j}^{a_{j+1}}  \frac{1}{(x-y)^{s+1}}  \, dy\,dx \\
&\le& \frac{8}{s(1-s)} \max\{1,\diam(A)\}.
\end{eqnarray*}
Hence we obtain \eqnref{domconv} from \eqnref{sum} by combining these estimates.
\end{proof}

A sequence of Borel sets $E_i\subset \rn$ converges to a Borel set $E\subset \rn$  if  $\one_{E_i}\to \one_E$ in $L^1(\rn)$, where $\one_E$ denotes the indicator function of $E$. The following lemma is the one-dimensional case of  \cite[Lemma 7]{AmbrosioDePhilippisMartinazzi} combined with the one-dimensional case of  \cite[Lemma 9]{AmbrosioDePhilippisMartinazzi} by Ambrosio, De Philippis \& Martinazzi.

\begin{lemma}\label{gamma} If $s_i\to 1^-$ and $A_i, A\subset \R$ are bounded Borel sets, then
\begin{equation*}
\liminf_{i\to \infty}(1-s_i) P_{s_i}(A_i)\ge \card(\rbd A)
\end{equation*}
for $A_i\to A$.
\end{lemma}

The following lemma contains the one-dimensional case of \eqnref{mazyashaposhnikova} for sets of finite perimeter and some estimates. It follows from \eqnref{mazyashaposhnikova} that \eqnref{conv0} also holds for bounded Borel sets of finite $s'$-perimeter for all $s'\in (0,1)$.

\begin{lemma}\label{two} If $A\subset \R$ is a bounded Borel set of finite perimeter, then
\begin{equation}\label{conv0}
\lim_{s\to 0^+}s\, P_s(A)= 2\,\Vol{A}
\end{equation}
and
\begin{equation}\label{domconv0}
P_s(A) \le  \frac{4}s\,\max\{1,\diam(A)\} +\diam(A)^2 +P_{s'}(A)
\end{equation}
for $\,0<s<s'< 1/2$. 
\end{lemma}

\begin{proof} 
Let $a=\inf A$ and $b=\sup A$. First, note that
$$\int_{-\infty}^{a} \int_{a}^{b} \frac1{\abs{x-y}^{1+s}}\,dx\,dy =  \frac{(b-a)^{1-s}}{s(1-s)}\le \frac 2 s \max\{1,\diam(A)\}$$
and
$$\int_{b}^{\infty} \int_{a}^{b} \frac1{\abs{x-y}^{1+s}}\,dx\,dy\le \frac 2 s \max\{1,\diam(A)\}.$$
Let $C= A^c \cap (a,b)$. Note that
$$\int\hskip -30pt\int\limits_{\{\abs{x-y}\ge 1\} \cap (A\times C)}  \frac1{\abs{x-y}^{1+s}}\,dx\,dy\,\, \le\,\, \diam(A)^2$$
and
$$\int\hskip -30pt\int\limits_{\{\abs{x-y}< 1\} \cap (A\times C)}  \frac1{\abs{x-y}^{1+s}}\,dx\,dy \,\,\,\le\,\,\,
\int\hskip -30pt\int\limits_{\{\abs{x-y}< 1\} \cap (A\times C)}  \frac1{\abs{x-y}^{1+s'}}\,dx\,dy. $$
Thus  \eqnref{domconv0} holds. 
\goodbreak

Next, we prove \eqnref{conv0}. 
Since $A$ has finite perimeter, it is the disjoint union of finitely many intervals lying at mutually positive distance up to a set of measure zero (cf.~\cite[Proposition~12.13]{Maggi}). 
Hence, w.l.o.g., $A= \bigcup_{i=1}^m I_i$, where $I_i=(a_i,b_i)$. Set $J_j= [b_j, a_{j+1}]$ for $j=1,\dots, m-1$ and $J_0=(-\infty, a_1)$ and $J_m=(b_m,\infty)$. We have
\begin{equation}\label{sum0}
P_s(A)=
 \sum_{j=0}^m \sum_{i=1}^m \int\limits_{J_j} \int\limits_{I_i}  \frac{1}{\abs{x-y}^{s+1}}  \, dx\,dy .
\end{equation}
A simple calculation shows that
\begin{equation*}
\lim_{s\to0^+} s \, \int_{J_j} \int_{I_i}  \frac{1}{\abs{x-y}^{s+1}}  \, dx\,dy =\Vol{I_i}
\end{equation*}
for $j=0$ and $j=m$ and
\begin{equation*}
\lim_{s\to0^+} s\,\int_{J_j} \int_I  \frac{1}{\abs{x-y}^{s+1}}  \, dx\,dy =0
\end{equation*}
for $j=1, \dots, m-1$.
Hence we obtain \eqnref{conv0} from \eqnref{sum0}.
\end{proof}

\section{Limiting Behavior of Fractional Perimeters}\label{lb}

Let $K\subset\rn$ be an origin-symmetric convex body.

\begin{theorem}\label{perimeter}
If $\,E\subset\rn$ is a bounded Borel set of finite perimeter, then
\begin{equation*}\label{perlim}
(1-s)\,P_s(E, K) \to P(E, \om K)
\end{equation*}
as $s\to 1^-$.
\end{theorem}

\begin{proof}
By the Blaschke-Petkantschin formula \eqnref{BP},
\begin{align*}
\int\limits_E\int\limits_{E^c}&\frac1{\norm{x-y}_K^{n+s}} \,dx\,dy \\
&= 
\int\limits_{E\cap L\ne \emptyset} \frac1{\norm{u(L)}_K^{n+s}} \int\limits_{E\cap L} \int\limits_{E^c\cap L} \frac1{\abs{x-y}^{s+1}} \,dH^1(x)\,dH^1(y)\,dL.
\end{align*}
Let $L_u=\{\lambda\, u: \lambda\in\R\}$. The sets $E\cap L$ have for $L=L_u +y$ for a.e.\ $y\in u^\bot$ finite perimeter (cf.\ \cite[Proposition 14.5]{Maggi}).
Hence we obtain by the Dominated Convergence Theorem, which can be used  because of \eqnref{domconv},  and by Lemma \ref{one} that
$$\lim_{s\to1^-} (1-s)
\int\limits_E\int\limits_{E^c}\frac1{\norm{x-y}_K^{n+s}} \, dx\,dy =
\int\limits_{E\cap L\ne \emptyset} \frac{\card(\rbd E\cap L)}{\norm{u(L)}_K^{n+1}}\,dL.$$
Since $\rbd E \cap L= \rbd (E\cap L)$ for a.e.\ line $L$ (cf.\ \cite[Theorem 18.11 and Remark 18.13]{Maggi}) and by the definition of the measure on the affine Grassmannian \eqnref{ag}, we get 
\begin{align*}
\int\limits_{E\cap L\ne \emptyset}&\frac{\card(\rbd E\cap L)}{\norm{u(L)}_K^{n+1}} \,dL\\
&=\frac12
\int\limits_{\sn} \int\limits_{E| u^\bot} \frac{\card(\rbd E \cap (y+ L_u))}{\norm{u}_K^{n+1}} \,dH^{n-1}(y)\,dH^{n-1}(u),
\end{align*}
where $L_u=\{\lambda\, u: \lambda\in\R\}$.
By \eqnref{proj}, Fubini's Theorem and the definition \eqnref{moment} of the moment body of $K$, we conclude that
\begin{align*}\lim_{s\to1^-} (1-s)
&\int\limits_E\int\limits_{E^c}\frac1{\norm{x-y}_K^{n+s}}\, dx\,dy \\
&=\frac12\int\limits_{\sn}\int\limits_{\rbd E} \frac{\abs{u\cdot \un_E(x)}}{\norm{u}_K^{n+1}}\, dH^{n-1}(x)\,dH^{n-1}(u)\\
&=\int\limits_{\rbd E} \norm{\nu_E(x)}_{\omd K}\, dH^{n-1}(x).
\end{align*}
The last term is the anisotropic perimeter of $E$ with respect to $\om K$.
\end{proof}

As a consequence, we obtain the following result. Combined with Theorem \ref{perimeter} we obtain Gamma-convergence of $(1-s) P_s(\cdot, K)$ to $P(\cdot, \om K)$ as $s\to 1^-$.

\begin{corollary}\label{gammainf}
Let $\,E_i, E\subset\rn$ be bounded Borel sets of finite peri\-meter. If $s_i\to 1^-$ and $E_i\to E$ as $i\to\infty$, then
\begin{equation*}\label{perliminf}
\liminf_{i\to\infty} \,(1-s_i)\,P_{s_i}(E_i, K) \ge  P(E, \om K).
\end{equation*}
\end{corollary}

\begin{proof}
By the Blaschke-Petkantschin formula \eqnref{BP}, Fatou's lemma and Lemma \ref{gamma},
\begin{align*}
\liminf_{i\to\infty}\, (1-s_i)\,&\int\limits_{E_i}\int\limits_{E_i^c}\frac1{\norm{x-y}_K^{n+s_i}} \,dx\,dy \\
&= \liminf_{i\to\infty} \int\limits_{E_i\cap L\ne \emptyset} 
\frac{ (1-s_i)\,P_{s_i}(E_i\cap L)}{\norm{u(L)}_K^{n+s_i}} \,dL\\
&\ge \int\limits_{E\cap L\ne \emptyset} \frac{\card(\rbd E \cap L)}{\norm{u(L)}_K^{n+1}}\,dL \,\,= \,\,P(E, \om K),
\end{align*}
where the last step is as in the proof of Theorem \ref{perimeter}.
\end{proof}

\goodbreak
The following theorem establishes the limiting behavior of anisotropic $s$-perimeters as $s\to 0$. Using the one-dimensional case of the result by Maz$'$ya \& Shaposhnikova \cite{MazyaShaposhnikova}, the theorem can also be derived for bounded Borel sets of finite $s'$-perimeters for all $s'\in(0,1)$.

\begin{theorem}\label{volume}
If $\,E\subset\rn$ is a bounded Borel set of finite perimeter, then
$$s\,P_s(E, K) \to n\, \Vol{K}\, \Vol{E}$$
as $s\to 0^+$.
\end{theorem}

\begin{proof}We proceed as in the proof of Theorem \ref{perimeter}. The Blaschke-Petkantschin formula \eqnref{BP} implies that
\begin{align*}
\int\limits_E\int\limits_{E^c}&\frac1{\norm{x-y}_K^{n+s}}\, dx\,dy \\
&=
\int\limits_{E\cap L\ne \emptyset} \frac1{\norm{u(L)}_K^{n+s}}\,\int\limits_{E\cap L} \int\limits_{E^c\cap L} \frac1{\abs{x-y}^{s+1}} \,dH^1(x)\,dH^1(y)\,dL.
\end{align*}
The sets $E\cap L$ have for $L=L_u +y$ for a.e.\ $y\in u^\bot$ finite perimeter (cf.~\cite[Proposition 14.5]{Maggi}).
Hence we obtain by the Dominated Convergence Theorem, which can be used  because of \eqnref{domconv0},  and Lemma \ref{two} that
$$\lim_{s\to0^+} s\,
\int\limits_E\int\limits_{E^c}\frac1{\norm{x-y}_K^{n+s}}\, dx\,dy =
2\!\!\!\int\limits_{E\cap L\ne \emptyset}  \frac{\Vol{ E\cap L}}{\norm{u(L)}_K^{n}}\, dL.$$
By the definition of the measure on the affine Grassmannian \eqnref{ag} and the polar coordinate formula for volume, we get  
\begin{eqnarray*}
2\!\!\!\int\limits_{E\cap L\ne \emptyset}  \frac{\Vol{E\cap L}}{\norm{u(L)}_K^{n}} \,dL&=&
\int\limits_{\sn} \int\limits_{E\cap u^\bot}  \frac{\Vol{E \cap (L_u+y)}}{\norm{u}_K^{n}} \,dH^{n-1}(y)\,dH^{n-1}(u)\\
&=&
\Vol{E}\!\!\int\limits_{\sn}  \frac1{\norm{u}_K^{n}} \,du\\
&=& n\,\Vol{K}\,\Vol{E}.
\end{eqnarray*}
This concludes the proof of the theorem.\end{proof}

\section{Fractional Isoperimetric Inequalities}\label{isoperi}

The next theorem shows that minimizers of the anisotropic fractional $s$-isoperimetric inequality with respect to $K$ converge as $s\to 1^-$ to minimizers of the  anisotropic isoperimetric inequality with respect to $\om K$.
\goodbreak

\begin{theorem}
Let $E_{s_i}\subset\rn$ be  bounded Borel sets such that
$$P_{s_i}(E_{s_i},K)=\gamma_{s_i}(K)\,\Vol{E_{s_i}}^{\frac{n-s_i}{n}}$$
and let $E_1\subset \rn$ be a bounded Borel set.  
If $s_i\to 1^-$ and $E_{s_i} \to E_1$  as $i\to \infty$, then there exists $c\ge 0$ such that
$E_1 = c \om K$ up to a set of measure zero.
\end{theorem}

\begin{proof}
If $E_1$ has measure zero, the statement is true for $c=0$. So, w.l.o.g., let $\Vol{E_1}=\Vol{\om K}$. Assume that $E_1$ is not a multiple of $\om K$ (up to a set of measure zero). Hence, by the equality case of the generalized Minkowski inequality \eqnref{minkowski} and Corollary \ref{gammainf}, we 
have
\begin{eqnarray*}
n\,\Vol{\om K} &<& P(E_1, \om K) \\
&\le& \liminf_{i\to \infty} \,(1-s_i)\, P_{s_i}(E_{s_i}, K)\\
&=& \liminf_{s\to 1^-} (1-s) \,\gamma_s(K)\, \Vol{\om K}^{\frac{n-s}n}.
\end{eqnarray*}
But 
\begin{eqnarray*}
\liminf_{s\to 1^-} (1-s)\, \gamma_s(K)\, \Vol{\om K}^{\frac{n-s}n} &\le&
\liminf_{s\to1^-} (1-s) P_s(\om K, K)\\
&=& P(\om K, \om K)\\
&=& n\, \Vol{\om K}.
\end{eqnarray*}
This is a contradiction. Thus $E_1$ is (up to a set of measure zero) a multiple of $\om K$.
\end{proof}
\goodbreak

\section{Preliminaries on Fractional Sobolev Norms}\label{norms}

For a function $f\in L^{1}(\rn)$ and $0<s<1$,  Gagliardo \cite{Gagliardo} 
introduced the fractional Sobolev $s$-seminorm of $f$ as
\begin{equation}\label{gagl}
\norm{f}_{W^{s,1}}= \int_\rn\int_\rn \frac{\abs{f(x)-f(y)}}{\abs{x-y}^{n+s}} \,dx\,dy.
\end{equation}
Extending a result by Bourgain, Brezis \& Mironescu \cite{BourgainBrezisMironescu} from $W^{1,1}(\rn)$ to $BV(\rn)$,  D\'avila~\cite{Davila} proved that for $f\in BV(\rn)$,
\begin{equation}\label{dav}
\lim_{s\to1^-} (1-s)  \norm{f}_{W^{s,1}} = 2 \,\alpha_{n}\, \norm{ f}_{BV},
\end{equation}
where $\alpha_{n}$ is the constant from \eqnref{davila},  the vector valued Radon measure $Df$ is the weak gradient of $f$, 
and $\norm{ f}_{BV}$ is the total variation of $Df$ on $\rn$. Note that 
\begin{equation}\label{bvnorm}
\norm{ f}_{BV}=  \int_{\rn} \babs{\frac{Df}{\abs{Df}}}\, d\abs{Df},
\end{equation}
where the vector $Df/\abs{Df}$ is the Radon-Nikodym derivative of $Df$ with respect to the total variation $\abs{Df}$ of $Df$. Also note that
\begin{equation}\label{coareaBV}
\norm{ f}_{BV}= \int_{0}^{\infty} P(\{\abs{f}>t\})\,dt
\end{equation}
by the coarea formula on $BV(\rn)$ (cf.\ \cite[Theorem 3.40]{AmbrosioFuscoPallara}).
\goodbreak

An anisotropic Sobolev seminorm on $BV(\rn)$ is defined by replacing the Euclidean norm  by an arbitrary norm in \eqnref{bvnorm}. For $K$ an origin-symmetric convex body in $\rn$, we set
\begin{equation*}\label{asobnorm} 
\norm{ f}_{BV, K} = \int_{\rn} \bnorm{\frac{Df}{\abs{Df}}}_{K^*}\, d\abs{Df}.
\end{equation*}
Note that 
\begin{equation}\label{acoarea} 
\norm{ f}_{BV, K} = \int_{0}^{\infty} P(\{\abs{f}>t\}, K)\, dt
\end{equation}
by the coarea forula on $BV(\rn)$ (cf.\ \cite[(2.22)]{FigalliMaggiPratelli10}).
Define the anisotropic fractional $s$-seminorm as 
\begin{equation*}
\int_\rn\int_\rn \frac{\abs{f(x)-f(y)}}{\norm{x-y}_K^{n+s}} \,dx\,dy,
\end{equation*}
where $K$ is an origin-symmetric convex body in $\rn$.

Visintin \cite{Visintin_siam} pointed out that as a consequence of Fubini's theorem a generalized coarea formula for fractional perimeters can be established. If $f\in L^1(\rn)$, then
\begin{equation*}
\int_{\rn}\int_{\rn} \frac{\abs{f(x)-f(y)}}{\norm{x-y}_K^{n+s}} \,dx\,dy= 
 2 \int_{-\infty}^\infty P_s(\{{f}>t\},K)\,dt
\end{equation*}
(or see \cite[Lemma 10]{AmbrosioDePhilippisMartinazzi}). 
If $K$ is origin-symmetric, then $P_s(E,K)= P_s(E^c,K)$ for all Borel sets $E\subset \rn$. Since $\Vol{\{f=t\}}=0$ a.e.\ on $\R$, we have $P_s(\{f\le -t\},K)= P_s(\{f< -t\}, K\})$ a.e.\ on $\R$. Hence 
\begin{equation}\label{visintin2}
\begin{array}{l}
\displaystyle\int_{\rn}\int_{\rn} \frac{\abs{f(x)-f(y)}}{\norm{x-y}_K^{n+s}} \,dx\,dy\\[12pt]
\displaystyle = 
 2 \int_{0}^\infty P_s(\{{f}>t\},K)\,dt +  2 \int_{0}^\infty P_s(\{{f}\le -t\},K)\,dt\\[12pt]
\displaystyle = 
 2 \int_{0}^\infty P_s(\{\abs{f}>t\},K)\,dt .
\end{array}
\end{equation}
In Section \ref{sobineq}, we make use of the Minkowski inequality for integrals in the following form: If $g: \rn\times \R\to [0,\infty)$ is measurable and $r>1$, then
\begin{equation}\label{MinkIneq}
\int_\R \big( \int_\rn g(x,t)^r\,dx\big)^\frac1r dt\ge \Big( \int_\rn \big( \int_\R g(x,t)\,dt\big)^r dx\Big)^\frac1r,
\end{equation}
If both sides are finite, there is equality if and only if $g(x,t)=\phi(x)\,\psi(t)$ a.e. with $\phi, \psi$ non-negative and measurable (cf.~\cite[(6.13.9)]{LittlewoodPolyaHardy}).

\section{Limits of Fractional Sobolev Norms}

For functions of bounded variation, we obtain the following analogue of the result \eqnref{dav} by D\'avila. 
Let $K\subset\rn$ be an origin-symmetric convex body.

\begin{theorem}\label{sobo}
If $f\in BV(\rn)$ has compact support, then
\begin{equation}\label{sobolim}
(1-s)\,\int_{\rn}\int_{\rn} \frac{\abs{f(x)-f(y)}}{\norm{x-y}^{n+s}_{K}} \,dx\,dy\to 2 \, \BVZK{f}
\end{equation}
as $s\to 1^-$.
\end{theorem}

\begin{proof}
By  the generalized coarea formula \eqnref{visintin2}, we obtain
$$\int_{\rn}\int_{\rn} \frac{\abs{f(x)-f(y)}}{\norm{x-y}_K^{n+s}} \,dx\,dy= 
2 \int_{0}^\infty P_s(\{\abs{f}>t\},K)\,dt.$$
By Lemma \ref{one} combined with the Blaschke-Petkantschin Formula \eqnref{BP} and the definition of the measure  on the affine Grassmannian \eqnref{ag}, we have
$$(1-s) P_s(E,K)\le 4 n \Vol{B} \max\{1,\diam(E)\}  \max\{1, \diam(K)\}^{n+1} P(E)$$
for $1/2\le s< 1$.
Hence 
\begin{equation}\label{vergl}
\begin{array}{ll}
 \displaystyle(1-s) &\displaystyle\!\!\!\int\limits_{0}^\infty P_s(\{\abs{f}>t\},K)\,dt \\
&\displaystyle\le \alpha(K) \max\{1,\diam(S)\} \int\limits_{0}^\infty P(\{\abs{f}>t\})\,dt,
\end{array}
\end{equation}
where $S$ is the support of $f$ and $\alpha(K)$ only depends on $K$. Since the function $f\in BV(\rn)$, the coarea formula \eqnref{coareaBV} implies that the right side of \eqnref{vergl} is finite.
Thus we conclude by the Dominated Convergence Theorem and Theorem \ref{perimeter} that
$$(1-s)\int_{\rn}\int_{\rn} \frac{\abs{f(x)-f(y)}}{\norm{x-y}_K^{n+s}} \,dx\,dy \to
 2 \int_{0}^\infty \!\!\!P(\{\abs{f}>t\},\om K)\,dt$$
as $s\to 1^-$. 
Combined with the coarea formula \eqnref{acoarea}, this concludes the proof of the theorem.
\end{proof}

\section{Fractional Sobolev Inequalities}\label{sobineq}

Let $K\subset\rn$ be an origin-symmetric convex body and $0<s<1$. Let $W^{s,1}(\rn)$ denote the set of  $f\in L^1(\rn)$ such that $\norm{f}_{W^{s,1}}< \infty$.

\goodbreak
\begin{theorem}\label{fracsob}
If $f\in W^{s,1}(\rn)$ has compact support, then
\begin{equation}\label{fracsobo}
\int\limits_{\rn} \int\limits_{\rn} \frac{\abs{f(x)-f(y)}}{\norm{x-y}_K^{n+s}}\,dx\,dy \ge 2\, \gamma_s(K)\, \Big(\int\limits_\rn \abs{f(x)}^{\frac n{n-s}} \,dx\Big)^{\frac {n-s}n}.
\end{equation}
There is equality in this inequality if and only if $f$ is a constant multiple of the indicator function of a minimizer of  \eqnref{afraciso}.
\end{theorem}

\goodbreak

\begin{proof}
If  $f\in W^{s,1}(\rn)$ has compact support, then by the generalized coarea formula \eqnref{visintin2} we obtain that
$$\int_{\rn}\int_{\rn} \frac{\abs{f(x)-f(y)}}{\norm{x-y}_K^{n+s}} \,dx\,dy\\[4pt]
= 
 2 \int_{0}^\infty P_s(\{\abs{f}>t\},K)\,dt.
$$
Hence, the isoperimetric inequality \eqnref{afraciso} and  the Minkowski inequality for integrals \eqnref{MinkIneq} imply that
\begin{align*}
\int_{\rn}\int_{\rn} &\frac{\abs{f(x)-f(y)}}{\norm{x-y}_K^{n+s}} \,dx\,dy \\
&\ge 2\,  \gamma_s(K)\, \int_{0}^\infty \Vol{\{\abs{f}>t\}}^{\frac{n-s}n} \,dt \\
&=  2\,\gamma_s(K)\, \int_{0}^\infty \Big( \int_\rn \one_{\{\abs{f}>t\}} (x)\,dx \Big)^{\frac{n-s}n}\,dt\\
&\ge  2 \,\gamma_s(K)\, 
\Big(\int_\rn\Big( \int_{0}^\infty \one_{\{\abs{f}>t\}} (x) \,dt\Big)^{\frac n{n-s}} \,dx\Big)^{\frac {n-s}n}\\
&=  2\, \gamma_s(K)\, \Big(\int_\rn \abs{f(x)}^{\frac n{n-s}} \,dx\Big)^{\frac {n-s}n}.
\end{align*}
This concludes the proof of the inequality.

Suppose there is equality in \eqnref{fracsobo}.
By the equality condition of \eqnref{MinkIneq}, we have $\one_{\{\abs{f}>t\}} (x)=\phi(x)\,\psi(t)$ with non-negative measurable functions $\phi, \psi$. Thus $f$ is a constant multiple of an indicator function. Since there is equality in \eqnref{afraciso}, we obtain that $f$ is a constant multiple of the indicator function of a minimizer of  \eqnref{afraciso}. On the other hand, if $f= c\,\one_{E_s}$, where $E_s$ is a minimizer of \eqnref{afraciso}, then
$$\int_{\rn}\int_{\rn} \frac{\abs{\one_{E_s}(x)-\one_{E_s}(y)}}{\norm{x-y}_K^{n+s}} \,dx\,dy 
= 2\, P(E_s, K).$$
Hence there is equality in \eqnref{fracsobo} if and only if $E_s$ is a minimizer of \eqnref{afraciso}.
\end{proof}

\section{Acknowledgements}

The author would like to thank Tuo Wang and the referees for their helpful comments.

\bigskip
\bigskip

\normalsize
\begin{samepage}
\noindent
Institut f\"ur Diskrete Mathematik und Geometrie\\
Technische Universit\"at Wien\\
Wiedner Hauptstr.\ 8-10/1046\\
1040 Wien, Austria\\
Email: monika.ludwig@tuwien.ac.at
\end{samepage}


\footnotesize

\begin{thebibliography}{10}

\bibitem{AlmgrenLieb}
F.~J. Almgren, Jr. and E.~H. Lieb, {\em Symmetric decreasing rearrangement is
  sometimes continuous}, J. Amer. Math. Soc. {\bf 2} (1989), 683--773, MR1002633, Zbl 0688.46014.

\bibitem{AlvinoFeroneTrombettiLions}
A.~Alvino, V.~Ferone, G.~Trombetti, and P.-L. Lions, {\em Convex symmetrization
  and applications}, Ann. Inst. H. Poincar\'e Anal. Non Lin\'eaire {\bf 14}
  (1997), 275--293, MR1441395, Zbl 0877.35040.

\bibitem{AmbrosioDePhilippisMartinazzi}
L.~Ambrosio, G.~De~Philippis, and L.~Martinazzi, {\em Gamma-convergence of
  nonlocal perimeter functionals}, Manuscripta Math. {\bf 134} (2011),
  377--403, MR2765717, Zbl 1207.49051.

\bibitem{AmbrosioFuscoPallara}
L.~Ambrosio, N.~Fusco, and D.~Pallara, {\em Functions of bounded variation and
  free discontinuity problems}, Oxford Mathematical Monographs, The Clarendon
  Press Oxford University Press, New York, 2000, MR1857292, Zbl 0957.49001.

\bibitem{BourgainBrezisMironescu}
J.~Bourgain, H.~Brezis, and P.~Mironescu, {\em Another look at {S}obolev
  spaces}, In: {O}ptimal Control and Partial Differential Equations ({J}. {L}.
  {M}enaldi, {E}. {R}ofman and {A}. {S}ulem, eds.). {A} volume in honor of {A}.
  {B}ensoussans's 60th birthday, {Amsterdam: IOS Press; Tokyo: Ohmsha}, 2001, Zbl 1103.46310.

\bibitem{BourgainBrezisMironescu02}
J.~Bourgain, H.~Brezis, and P.~Mironescu, {\em Limiting embedding theorems for
  {$W\sp {s,p}$} when {$s\uparrow1$} and applications}, J. Anal. Math. {\bf 87}
  (2002), 77--101, Dedicated to the memory of Thomas H. Wolff, MR1945278, Zbl 1029.46030.

\bibitem{CaffarelliRoquejoffreSavin}
L.~Caffarelli, J.-M. Roquejoffre, and O.~Savin, {\em Nonlocal minimal
  surfaces}, Comm. Pure Appl. Math. {\bf 63} (2010), 1111--1144, MR2675483, Zbl 1248.53009.

\bibitem{CaffarelliValdinoci}
L.~Caffarelli and E.~Valdinoci, {\em Uniform estimates and limiting arguments
  for nonlocal minimal surfaces}, Calc. Var. Partial Differential Equations
  {\bf 41} (2011), 203--240, MR2782803, Zbl 05884582.

\bibitem{Cordero:Nazaret:Villani}
D.~Cordero-Erausquin, B.~Nazaret, and C.~Villani, {\em A mass-transportation
  approach to sharp {S}obolev and {G}agliardo-{N}irenberg inequalities}, Adv.
  Math. {\bf 182} (2004), 307--332, MR2032031, Zbl 1048.26010.

\bibitem{Davila}
J.~D{\'a}vila, {\em On an open question about functions of bounded variation},
  Calc. Var. Partial Differential Equations {\bf 15} (2002), 519--527, MR1942130, Zbl 1047.46025.

\bibitem{DipierroFigalliPalatucciValdinoci}
S.~Dipierro, A.~Figalli, G.~Palatucci, and E.~Valdinoci, {\em Asymptotics of
  the s-perimeter as $s \searrow 0$}, Discrete Contin. Dyn. Syst. {\bf 33} (2013), 2777--2790.

\bibitem{FigalliMaggiPratelli10}
A.~Figalli, F.~Maggi, and A.~Pratelli, {\em A mass transportation approach to
  quantitative isoperimetric inequalities}, Invent. Math. {\bf 182} (2010),
  167--211, MR2672283, Zbl 1196.49033.

\bibitem{FigalliMaggiPratelli}
A.~Figalli, F.~Maggi, and A.~Pratelli, {\em Sharp stability theorems for the
  anisotropic {S}obolev and log-{S}obolev inequalities on functions of bounded
  variation}, Adv. Math. {\bf  242} (2013), 80--101, MR3055988.

\bibitem{FrankSeiringer}
R.~Frank and R.~Seiringer, {\em Non-linear ground state representations and
  sharp {H}ardy inequalities}, J. Funct. Anal. {\bf 255} (2008), 3407--3430, MR2469027, Zbl 1189.26031.

\bibitem{FuscoMillotMorini}
N.~Fusco, V.~Millot, and M.~Morini, {\em A quantitative isoperimetric
  inequality for fractional perimeters}, J. Funct. Anal. {\bf 261} (2011),
  697--715, MR2799577, Zbl 1228.46030.

\bibitem{Gagliardo}
E.~Gagliardo, {\em Caratterizzazioni delle tracce sulla frontiera relative ad
  alcune classi di funzioni in {$n$} variabili}, Rend. Sem. Mat. Univ. Padova
  {\bf 27} (1957), 284--305, MR0102739, Zbl 0087.10902.

\bibitem{Gardner}
R.~Gardner, {\em Geometric {T}omography}, second ed., Encyclopedia of
  Mathematics and its Applications, vol.~58, Cambridge University Press,
  Cambridge, 2006, MR2251886, Zbl 1102.52002.

\bibitem{Gromov_Sobolev}
M.~Gromov, {\em Isoperimetric inequalities in {R}iemannian manifolds},
  {Appendix of: Asymptotic Theory of Finite-dimensional Normed Spaces (V. D.
  Milman \& G. Schechtman)}, Springer-Verlag, 1986, MR0856576, Zbl 0606.46013.

\bibitem{Gruber}
P.~M. Gruber, {\em Convex and {D}iscrete {G}eometry}, Grundlehren der
  Mathematischen Wissenschaften, vol.~336, Springer, Berlin, 2007, MR2335496, Zbl 1139.52001.

\bibitem{Haberl_sln}
C.~Haberl, {\em Minkowski valuations intertwining the special linear group}, J.
  Eur. Math. Soc. (JEMS) (2012), 1565--1597, MR2966660, Zbl 06095891.

\bibitem{Haberl:Schuster1}
C.~Haberl and F.~Schuster, {\em General ${L}_p$ affine isoperimetric
  inequalities}, J. Differential Geom. {\bf 83} (2009), 1--26, MR2545028, Zbl 1185.52005.

\bibitem{LittlewoodPolyaHardy}
G.~H. Hardy, J.~E. Littlewood, and G.~P{\'o}lya, {\em Inequalities}, Cambridge
  Mathematical Library, Cambridge University Press, Cambridge, 1988, Reprint of
  the 1952 edition, MR0944909, Zbl 0634.26008.

\bibitem{Ludwig:Minkowski}
M.~Ludwig, {\em Minkowski valuations}, Trans. Amer. Math. Soc. {\bf 357}
  (2005), 4191--4213, MR2159706, Zbl 1077.52005.

\bibitem{Ludwig:convex}
M.~Ludwig, {\em Minkowski areas and valuations}, J. Differential Geom. {\bf 86}
  (2010), 133--161, MR2772547, Zbl 1215.52004.

\bibitem{Lutwak90}
E.~Lutwak, {\em Centroid bodies and dual mixed volumes}, Proc. London Math.
  Soc. {\bf 60} (1990), 365--391, MR1031458, Zbl 0703.52005.

\bibitem{LYZ2000}
E.~Lutwak, D.~Yang, and G.~Zhang, {\em ${L}\sb p$ affine isoperimetric
  inequalities}, J. Differential Geom. {\bf 56} (2000), 111--132, 
MR1863023, Zbl 1034.52009.

\bibitem{LYZ2000b}
E.~Lutwak, D.~Yang, and G.~Zhang, {\em A new ellipsoid associated with convex
  bodies}, Duke Math. J. {\bf 104} (2000), 375--390, MR1781476, Zbl 0974.52008.

\bibitem{LYZ2002}
E.~Lutwak, D.~Yang, and G.~Zhang, {\em The {C}ramer-{R}ao inequality for star
  bodies}, Duke Math. J. {\bf 112} (2002), 59--81,  MR1890647, Zbl 1021.52008.


\bibitem{LYZ2010a}
E.~Lutwak, D.~Yang, and G.~Zhang, {\em Orlicz centroid bodies}, J. Differential
  Geom. {\bf 84} (2010), 365--387, MR2652465, Zbl 0703.52005.

\bibitem{LZ1997}
E.~Lutwak and G.~Zhang, {\em Blaschke-{S}antal\'o inequalities}, J.
  Differential Geom. {\bf 47} (1997), 1--16, 
MR1601426, Zbl 0906.52003.

\bibitem{Maggi}
F.~Maggi, {\em Sets of finite perimeter and geometric variational problems},
  Cambridge Univ. Press, 2012, MR2976521, Zbl 1255.49074.

\bibitem{MazyaShaposhnikova}
V.~Maz$'$ya and T.~Shaposhnikova, {\em On the {B}ourgain, {B}rezis, and
  {M}ironescu theorem concerning limiting embeddings of fractional {S}obolev
  spaces}, J. Funct. Anal. {\bf 195} (2002), 230--238, 
MR1940355, Zbl 1028.46050.

\bibitem{Schneider:CB}
R.~Schneider, {\em Convex {B}odies: the {B}runn-{M}inkowski {T}heory},
  Encyclopedia of Mathematics and its Applications, vol.~44, Cambridge
  University Press, Cambridge, 1993, 
MR1216521, Zbl 1143.5200.

\bibitem{SchneiderWeil}
R.~Schneider and W.~Weil, {\em Stochastic and {I}ntegral {G}eometry},
  Probability and its Applications (New York), Springer-Verlag, Berlin, 2008, 
MR2327289, Zbl 1175.60003.

\bibitem{Taylor1978}
J.~Taylor, {\em Crystalline variational problems}, Bull. Amer. Math. Soc. {\bf
  84} (1978), 568--588, 
MR0493671, Zbl 0392.49022.

\bibitem{Visintin_siam}
A.~Visintin, {\em Nonconvex functionals related to multiphase systems},
  SIAM J. Math. Anal. {\bf 21} (1990),  1281--1304, MR106240, Zbl 0723.49006.

\bibitem{Tuo_Wang}
T.~Wang, {\em The affine {S}obolev-{Z}hang inequality on {BV}$({\R}^n)$}, Adv.
  Math. {\bf 230} (2012), 2457--2473, MR2927377, Zbl 1257.46016.

\bibitem{Wannerer2011}
T.~Wannerer, {\em ${\rm {GL}}(n)$ equivariant {M}inkowski valuations},
  Indiana Univ. Math. J. {\bf 60} (2011), 1655--1672, MR2997003, Zbl 06122681.

\bibitem{Wieacker84}
J.~Wieacker, {\em Translative {P}oincar\'e formulae for {H}ausdorff rectifiable
  sets}, Geom. Dedicata {\bf 16} (1984), 231--248, MR0758909, Zbl 0544.52004.

\bibitem{Zhang99}
G.~Zhang, {\em The affine {S}obolev inequality}, J. Differential Geom. {\bf 53}
  (1999), 183--202, 
MR1776095, Zbl 1040.53089.

\end{thebibliography}
\end{document}